\documentclass[a4paper,reqno]{amsart}
\usepackage{amsmath}
\usepackage{amsfonts}
\usepackage{amssymb}
\usepackage{xcolor}

\newtheorem{theorem}{Theorem}[section]
\newtheorem{lemma}[theorem]{Lemma}
\newtheorem{definition}[theorem]{Definition}

\newtheorem{corollary}[theorem]{Corollary}
\newtheorem{remark}[theorem]{Remark}
\newtheorem{example}[theorem]{Example}

\newcommand{\RR}{\mathbb{R}}
\newcommand{\FF}{\mathbb{F}}

\newcommand{\PP}{\mathbb{P}}

\newcommand{\E}{\mathcal{E}}

\newcommand{\G}{G(n,n-m)}
\newcommand{\A}{A(n,m)}

\newcommand{\F}{\mathbb{F}_{p}^{n}}

\newcommand{\ee}{e^{-\frac{2\pi i  x\cdot\xi}{p}}}

\title[Restricted families of projections]{Restricted families of projections in vector space over finite fields}

\author{Changhao Chen}

\address{School of Mathematics and Statistics, The University of New South Wales, Sydney NSW 2052, Australia }
\email{changhao.chenm@gmail.com}

\date{\today}


\thanks{The author acknowledges the support of the Vilho, Yrj\"o, and Kalle V\"ais\"al\"a foundation.}

\begin{document}

\begin{abstract}
We study the restricted families of projections in  vector spaces over finite fields. We show that there are  families of  random subspaces which admit a Marstrand-Mattila type projection theorem. 
\end{abstract}

\maketitle

\section{Introduction}
A fundamental problem in fractal geometry is to determine  how the projections affect dimension. Recall the classical Marstrand-Mattila projection theorem: Let $E\subset \RR^{n}, n\geq2,$ be a Borel set with Hausdorff dimension $s$. 
\begin{itemize}
\item (dimension part) If $s\leq m$, then the orthogonal projection of $E$ onto almost all $m$-dimensional subspaces has Hausdorff dimension $s$.

\item (measure part) If $s>m$, then the orthogonal of $E$ onto almost all $m$-dimensional subspaces has positive $m$-dimensional Lebesgue measure. 
\end{itemize}

In 1954 J. Marstand \cite{Marstrand} proved this projection theorem
in the plane. In 1975 	P. Mattila \cite{Mattila1975} 
proved this  for general dimension via   1968 R. Kaufman's  \cite{Kaufman} potential theoretic methods.   We refer to the recent survey of  K. Falconer, J. Fraser, and X. Jin \cite{Falconer}  for more backgrounds. 

Recently there has been a growing interest in studying finite field version of some classical problems arising from Euclidean spaces. For instance, there are finite field  Kakeya sets (also called Besicovitch sets), see Z. Dvir \cite{Dvir}; there are  finite field Erd\H{o}s/ Falconer distance problem, see  A. Iosevich, M. Rudnev \cite{IosevichRudnev}, T. Tao \cite{Tao1}; etc. Motivated by the above works, the author \cite{ChenP} studied the  projections in vector spaces over finite fields, and obtained the Marstrand-Mattila type projection theorem in this setting. In this paper, we turn to the restricted families of projections in the vector spaces over finite fields.  For more details on projection in vector space over finite fields see \cite{ChenP}. For more backgrounds on restricted families of projections in Euclidean spaces, we refer to \cite[Section 6]{Falconer}, \cite{FOO}, \cite{KOV} and reference therein.

Let $p$ be a prime number, $\mathbb{F}_{p}$ be the finite field with $p$ elements, and $\F$ be the $n$-dimensional vector space over this field. We use the same notation as in the Euclidean spaces. Let $G(n,m)$ be the collection of all $m$-dimensional linear subspaces of $\F$, and  $\A$ be the family of all $m$-dimensional planes, i.e., the translation of some $m$-dimensional subspace.  In the following we show the definition of projections in $\F$, see \cite{ChenP} for more details.
 
\begin{definition}
Let $E$ be a subset of  $\F$ and $W$ be a non-trivial subspace of $\F$. Denoted by $\pi^{W}(E)$ the collection of coset of $W$ which intersects $E$, that is 
\[
\pi^{W}(E)=\{x+W: E\cap (x+W) \neq \emptyset, x\in \F\}.
\]
In this paper we are interested in the cardinality of $\pi^{W}(E)$.  
\end{definition}

For any set $E \subset \F$ and $W\in G(n,n-m)$ the  Lagrange's group theorem implies 
\[
|\pi^{W}(E)|\leq \min\{|E|, p^{m}\}.
\]
Here and in the following let $|J|$ denote the cardinality of set $J$. The author \cite[Corollary 1.3]{ChenP}  obtained the following Marstrand-Mattila type projection theorem in $\F$. In fact the following form is often called the size of the exceptional sets of projections.

\begin{theorem}\label{thm:ChenP}
Let $E \subset \F$ with $|E|=p^{s}$. 

(a) If $s\leq m$ and $t \in (0, s]$, then
\[
| \{W \in  G(n,n-m) : |\pi^{W} (E)| \leq p^{t}/10 \} \leq \frac{1}{2} p^{m(n - m) -(m - t)}.
\]

(b) If $s> m$, then
\[
| \{W \in  G(n,n-m) : |\pi^{W} (E)| \leq   p^{m}/ 10 \}| \leq \frac{1}{2} p^{m(n - m) -(s-m)}.
\]
\end{theorem}
We note that  $|G(n,m)|\approx p^{m(n-m)}$, see P. Cameron \cite[Theorem 6.3]{Cameron}.  We write $f\lesssim g$ if there is a positive constant $C$ such that $f\leq Cg$, $f\gtrsim g$ if $g\lesssim f$, and $f\approx g$ if $f\lesssim g$ and $f\gtrsim g$.   

In the following, we formulate finite fields version of   restricted families of projections.  Let $G$ be a subset of $G(n,k)$, then $(\pi^{W})_{W\in G }$ is called a restricted family of projection. The purpose of this paper is looking for subsets $G\subset G(n,k)$ such that $(\pi^{W})_{W\in G }$ admit a Marstrand-Mattila type projection theorem.

By studying the  random subsets of $G(n, n-m)$, we  obtain the following result.  

\begin{theorem}\label{thm:main}
For any $  \min\{m, n-m\} <\alpha \leq  m(n-m)$, there exists a subset $G\subset G(n,n-m)$ with $|G|\approx p^{\alpha}$ such that  for any $E\subset \F$,
\begin{equation*}\label{eq:eee}
 |\{W\in G: |\pi^{W}(E)|\leq N\}|\lesssim |G|N(|E|^{-1}+p^{-m}).
\end{equation*} 
\end{theorem}
Note that for the case $\alpha=m(n-m)$, Theorem \ref{thm:main} follows from Theorem \ref{thm:ChenP} by choosing $G=G(n,m)$. Thus we consider the case $\min\{m, n-m\} <\alpha < m(n-m)$ only.  We immediately  have the following Marstrand-Mattila type projection theorem via the special choice $N$ in  Theorem \ref{thm:main}. 

\begin{corollary}
For any $  \min\{m, n-m\} <\alpha \leq  m(n-m)$, there exists a subset $G\subset G(n,n-m)$ with $|G|\approx p^{\alpha}$ such the following holds. Let $E\subset \F$ with $|E|=p^{s}$. 

(1) If $|E|\leq p^{m}$ and $t\in (0,s]$, then  
\begin{equation*}
 |\{W\in G: |\pi^{W}(E)|\leq p^{t}\}|\lesssim |G|p^{t-s}.
\end{equation*} 

(2) If $|E|> p^{m}$, then for any small $\varepsilon$
\begin{equation*}
 |\{W\in G: |\pi^{W}(E)|\leq \varepsilon p^{m}\}|\lesssim |G|\varepsilon.  
\end{equation*} 
\end{corollary}
 
For restricted families of projections in Euclidean spaces, 
the author \cite{ChenR} obtained that some random subsets of sphere of $\RR^{3}$ admit a Marstrand-Mattila type projection theorem. For more details, see \cite{ChenR}.

The structure of the paper is as follows. In Section \ref{sec:p}, we set up some notation and show some lemmas for later use. We prove Theorem  \ref{thm:main} in Section \ref{sec:proof}.  In the last Section we given some examples of restricted families of projections which admit a Marstrand-Mattila type theorem in finite fields setting.

\section{Preliminaries}\label{sec:p}

In this section we show some lemmas for later use. 

\subsection{Outline of the methods}

In short words, we take a random subset $G\subset G(n,n-m)$, see the random model in Subsection \ref{sub:rrr}. Then we estimate the cardinality of ``the exceptional sets'',
\[
\{W\in G: |\pi^{W} (E)|\leq N\},
\]
and show that it satisfies our need. To estimate the ``exceptional sets'', we adapt the arguments in \cite{ChenP} to our setting which is a variant of Orponen's pairs argument \cite[Estimate (2.1)]{OrponenA}.

Let $W\in G(n,n-m)$ then Lagrange’s group theorem implies that  there are $p^{m}$ cosets of $W$. Let $x_{W, j}+W, 1\leq j\leq p^{m}$
be the different cosets of $W$.  Let  $E\subset \FF_{p}^{n}$, then 
\[
|E| =\sum_{j=1}^{p^{m}} |E\cap (x_{W, j}+W)|,
\]
and the Cauchy-Schwarz inequality  implies  
\begin{equation}\label{eq:pairss}
|E|^{2}\leq |\pi^{W}(E)|\sum_{j=1}^{p^{m}} |E\cap (x_{W, j}+W)|^{2}.
\end{equation}
Note that $|E\cap (x_{W, j}+W)|^{2}$ is the amount of pairs of $E$ inside the coset $x_{W, j}+W$. Let $N\leq p^{m}$, define 
\[
\Theta=\{W\in G: |\pi^{W} (E)|\leq N\}.
\]
Summing two sides over 
$W\in \Theta$ in estimate \eqref{eq:pairss}, we obtain 
\begin{equation}\label{eq:argument}
|\Theta| |E|^{2} \leq \E(E,\Theta')N
\end{equation}
where $\E(E,\Theta')=\sum_{W\in \Theta}\sum_{j=1}^{p^{m}}|E\cap (x_{W, j}+W)|^{2}.$  
Thus the left problem is to estimate $\E(E, \Theta')$, and we use the doubling counting argument of Murphy and Petridis \cite[Lemma 1]{MurphyPetridis} and the discrete Plancherel identity. The above discusses motivated  the following definition.

%

\begin{definition}
Let $E\subset \F$ and $ \mathcal{A} \subset A(n,m)$. Define the  energy of $E$ on $\mathcal{A}$ as
\[
\E(E, \mathcal{A}) =\sum_{W\in \mathcal{A}} |E\cap W|^{2}.
\] 
\end{definition}
	
We note that $\E(E, \mathcal{A})$ is closely related to the incidence identity of Murphy and Petridis \cite[Lemma 1]{MurphyPetridis}, and the additive energy in additive combinatorics \cite[Chapter 2]{TaoVu}.

\subsection{Discrete Fourier transform}
In the following we collect some basic facts about Fourier transformation which related to our setting. For more details on discrete Fourier analysis, see Green \cite{Green}, Stein and Shakarchi \cite{Stein}. Let $f : \F\longrightarrow \mathbb{C}$ be a complex value function. Then for $\xi \in \F$ we define the Fourier transform 
\begin{equation*}\label{eq:dede}
\widehat{f}(\xi)=\sum_{x\in \F} f(x)e(-x\cdot \xi),
\end{equation*}  
where the dot product $ x\cdot\xi $ is defined as $ x_1\xi_1+\cdots +x_n\xi_n$ and $e(-x \cdot \xi)=\ee$. Recall the following Plancherel identity, 
\begin{equation*}
 \sum_{\xi \in \F}|\widehat{f}(\xi)|^{2}=p^{n}\sum_{x\in \F} |f(x)|^{2}.
\end{equation*} 
Specially for the subset $E\subset \F$, we have 
\[
\sum_{\xi \in \F} |\widehat{E}(\xi)|^{2}=p^{n}| E|.
\]
Here and in the following we use $E$ as characteristic function of the set  $E$.
 
 For $W\in G(n,n-m)$, define
\[
Per(W):=\{x\in \F: x\cdot w=0, w\in W\}.
\]
Note that if $W$ is some subspace in Euclidean space then $Per(W)$ is the orthogonal complement of $W$. 
Furthermore,  unlike in the Euclidean spaces,  here $W\cap Per(W)$ can be some non-trivial subspace.  However  the rank-nullity theorem of linear algebra implies  that for any subspace $W \subset\F$,
\begin{equation}\label{eq:rank}
\dim W+\dim Per(W)=n.
\end{equation}

The following Lemma \ref{lem:fff} of \cite[Lemma 2.3]{ChenP} plays an important role  in the proof of Lemma  \ref{lem:abstract} (2). For more details see \cite[Lemma 2.3]{ChenP}.

\begin{lemma} \label{lem:fff}
Use the above notation.  We have 
\begin{equation}\label{eq:kk}
\sum_{j=1}^{p^{m}} | E \cap (x_{j}+W)|^{2}=p^{-m}\sum_{\xi\in Per(W)} |\widehat{E}(\xi)|^{2}.
\end{equation}
\end{lemma}

\begin{remark}
We note that the Lemma \ref{lem:fff} is the only place in this paper where the prime field $\mathbb{F}_{p}$ is needed. We do not know if the Lemma \ref{lem:fff} also holds for vector spaces over general finite fields. 
\end{remark}

In the following we extend a result of \cite[Lemma 3.1]{ChenP} to general subset of $\G$. Let $G\subset G(n,n-m)$, define
\begin{equation}\label{eq:define}
G'=\bigcup_{W\in G}\bigcup_{j=1}^{p^{m}}(x_{j,W}+W)
\end{equation}
where $x_{W, j}+W, 1\leq j\leq p^{m}$
are the cosets of $W$. For each $W$ we simply use $x_{W, j}+W, 1\leq j\leq p^{m}$ to represent the cosets of $W$.  

\begin{lemma}\label{lem:abstract}
Let $G$ be a subset of $G(n, n-m)$ with $|G|\gtrsim p^{\beta}$. 

(1) If for any $\xi\neq 0$,  
\begin{equation}\label{eq:l11}
|\{W\in G: \xi \in V\}|\lesssim |G| p^{-\beta},
\end{equation}
then 
\begin{equation}
\E(E,G')\lesssim |E||G|+|E|^{2}|G|p^{-\beta}.
\end{equation}

(2) If for any $\xi\neq 0$,
\begin{equation}\label{eq:l22}
|\{W\in G: \xi \in Per(W)\}|\lesssim |G| p^{-\beta},
\end{equation}
then 
\begin{equation}
\E(E,G')\lesssim p^{-m}|G|(|E|^{2}+|E|p^{n-\beta}).
\end{equation}
\end{lemma}
\begin{proof}
The claim $(1)$ follows by doubling counting. Recall that we denote by $F(x)$ the characteristic function of the subset $F\subset \F$. Then
\begin{equation*}
\begin{aligned}
\E(E, G')&= \sum_{V\in G' }|E \cap V|^{2}\\
&=\sum_{V \in G'} \left(\sum_{x\in E}V(x) \right)^{2}\\
&=\sum_{V \in G'} \left(\sum_{x\in E}V(x)+\sum_{x\neq y \in E} V(x)V(y) \right)\\
&\lesssim |E||G|+|E|(|E|-1)|G|p^{-\beta}.
\end{aligned}
\end{equation*}

To establish  $(2)$, the Lemma \ref{lem:fff} implies 
\begin{equation*}
\begin{aligned}
\E(E, G')
&=\sum_{W\in G} \sum_{j=1}^{p^{m}}|E \cap (x_{W, j}+W)|^{2}\\
& =p^{-m}\sum_{W\in G} \sum_{\xi\in Per(W)}|\widehat{E}(\xi)|^{2}\\
& =p^{-m}(|G||E|^{2}+\sum_{W\in G} \sum_{\xi\in Per(W)\backslash \{0\}}|\widehat{E}(\xi)|^{2})\\
&\lesssim p^{-m}(|G||E|^{2}+p^{n}|E||G|p^{-\beta}).
\end{aligned}
\end{equation*}
Thus we finish the proof.
\end{proof}

\section{Proof of Theorem \ref{thm:main}}\label{sec:proof}

\subsection{Random subsets of $\G$}\label{sub:rrr}

We start by a description of these random subsets in $\G$. Let $0<\delta<1$. We choose each element of $\G$ with probability $\delta$ and remove it with probability $1-\delta$, all choices being independent of each other. Let $G=G^{\omega}$ be the collection of these chosen subspaces. Let $\Omega (\G, \delta)$ be our probability space which consists of all the possible outcomes of random subspaces.

For the convenience to our use, we formulate the following large deviations estimate. For more background and details on large deviations estimate, see Alon and Spencer \cite[Appendix A]{Alon}.

\begin{lemma}[Chernoff bound]\label{lem:law of large numbers}
Let $\{X_j\}_{j=1}^N$ be a sequence  independent Bernoulli random variables which takes value $1$ with probability $\delta$ and value $0$ with probability $1-\delta$. Then 
\[
\PP(  \sum^N_{j=1} X_j  \geq  3N\delta )\leq e^{-N\delta}.
\]
\end{lemma}

We also need the following Lemma of \cite[Lemma 2.7]{ChenP}. 
\begin{lemma} \label{lem:c}
Let $\xi$ be a non-zero vector of $\F$.

(1) $|\{W\in G(n, k): \xi\in W\}|=|G(n-1, k-1)|$. 

(2) $|\{W\in G(n, k): \xi\in Per(W)\}|=|G(n-1, k)|$.
\end{lemma}


\begin{corollary}\label{co:cc}
For any $m<\alpha<m(n-m)$, there exists a subset $G\subset \G$ such that
$|G|\approx p^{\alpha}$  and for any $\xi\neq 0$,
\[
|\{W\in G: \xi \in W\}|\lesssim  |G|p^{-m}.
\]
\end{corollary}
\begin{proof}
We consider the random model $\Omega(\G, \delta)$ where $\delta=|G(n,m)|^{-1}p^{\alpha}$.  First observe that $p^{\alpha}/2 \leq |G| \leq 2 p^{\alpha}$ with high probability ($>1/2$) provided large $p$. This follows by applying Chebyshev's inequality, which says that
\begin{equation}\label{eq:che}
\begin{aligned}
\PP(||G| - p^{\alpha}|&> \frac{1}{2}p^{\alpha})\leq \frac{4p^{\alpha}(1-\delta)}{p^{2\alpha}}\\
&\leq  \frac{4}{p^{\alpha}}\rightarrow 0 \text{ as } p \rightarrow \infty.
\end{aligned}
\end{equation}

Let $\xi\neq 0$ and $G_{\xi}:=\{W\in \G: \xi \in W\}$. Lemma \ref{lem:c} (1) implies that 
\[
|G_{\xi}|=|G(n-1,n-m-1)|\approx p^{m(n-m)-m}.
\]
Observe that for $G\in \Omega(\G,\delta)$,
\[
|\{W\in G: \xi \in W\}|=\sum_{W\in G_{\xi}}{\bf 1}_{G}(W).
\]
Thus by Lemma \ref{lem:law of large numbers}, 
\[
\PP(\sum_{W\in G_{\xi}}{\bf 1}_{G}(W)\geq 3|G(n-1,m)|\delta)\leq e^{-Cp^{\alpha-m}}
\]
where $C$ is a positive constant. It follows that 
\begin{equation*}
\begin{aligned}
\PP(\exists \xi\neq 0, s.t. \sum_{W\in G_{\xi}}&{\bf 1}_{G}(W) \geq 3|G(n-1,m)|\delta)\\
&\leq p^{n}e^{-Cp^{\alpha-m}}\rightarrow 0 \text{ as } p \rightarrow \infty.
\end{aligned}
\end{equation*}
Together with the estimate \eqref{eq:che}, we conclude that $G\in \Omega(\G, \delta)$ satisfies our need with high probability (at least one) provided $p$ is large enough.  
\end{proof}

\begin{corollary}\label{co:ccc}
For any $n-m<\alpha<m(n-m)$, there exists a subset $G\subset \G$ such that
$|G|\approx p^{\alpha}$  and for any $\xi\neq 0$, 
\[
|\{W\in G: \xi \in Per(W)\}|\lesssim  |G|p^{-(n-m)}.
\]
\end{corollary}
\begin{proof}
We consider the random model $\Omega(\G, \delta)$ where $\delta=|G(n,m)|^{-1}p^{\alpha}$. For any $\xi\neq 0$, Lemma \ref{lem:c} (2) implies that 
\[
|\{W\in G(n, n-m): \xi \in Per(W)\}|=|G(n-1, n-m)|\approx p^{m(n-m)-(n-m)}.
\]
Then applying the similar argument to the proof of Corollary \ref{co:cc}, we obtain  that $G\in \Omega(\G, \delta)$ satisfies our need
with high probability provided $p$ is large enough.
\end{proof}

Now we intend to apply Lemma \ref{lem:abstract} and the above two Corollaries to prove Theorem \ref{thm:main}.

\begin{proof}[Proof of Theorem \ref{thm:main}]
Suppose $\alpha>m$. By Corollary \ref{co:cc} there exists a subset $G\subset \G$ such that $|G|\approx p^{\alpha}$  and for any $\xi\neq 0$,
\[
|\{W\in G: \xi \in W\}|\lesssim  |G|p^{-m}.
\]
Applying Lemma \ref{lem:abstract} (1), we obtain that for any $E\subset \F$,
\[
\E(E,G')\lesssim |G|(|E|+|E|^{2}p^{-m}).
\]
By estimate \eqref{eq:argument} we obtain
\[
|\{W\in G: |\pi^{W}(E)|\leq N\}|\lesssim |G|N(|E|^{-1}+p^{-m}). 
\]

For the case $\alpha>n-m$. By Corollary \ref{co:ccc} there exists a subset $G\subset G(n,n-m)$ with $|G|\approx p^{\alpha}$ and for any $\xi\neq 0$,
\[
|\{W\in G: \xi \in Per(W)\}|\lesssim  |G|p^{n-m}.
\]
Applying Lemma \ref{lem:abstract} (2), we obtain that for any $E\subset \F$,
\[
\E(E,G')\lesssim |G|(|E|+|E|^{2}p^{-m}).
\]
Again by  estimate \eqref{eq:argument}, we obtain
\[
|\{W\in G: |\pi^{W}(E)|\leq N\}|\lesssim |G|N(|E|^{-1}+p^{-m}). 
\]
Thus we complete the proof. 
\end{proof}

\section{Examples}

We show two examples in the following. For $D\subset \F$  let $G_{D}$ be
the collection of one dimensional subspaces which intersects $D$, i.e.,
\[
G_{D}=\{kx: x\in D, k\in \mathbb{F}_{p}\}.
\]

\begin{example}\label{exa:1}
Let $
S_{1}=\{(x_{1}, x_{2}, 1)\in \FF_{p}^{3}: x_{1}^{2}+x_{2}^{2}=1\}$.  Then for any $E\subset \mathbb{F}_{p}^{3}$,
\[
|\{L\in G_{S_{1}}: |\pi^{L}(E)|\leq N\}|\lesssim |S_{1}|N(p^{-2}+|E|^{-1}).
\]
\end{example}
\begin{proof}
A. Iosevich and M. Rudnev \cite[Lemma 2.2]{IosevichRudnev} proved that $|S_{1}|\approx p$, and hence $|G_{S_{1}}|\approx p$. Observe that  $|W\cap S_{1}|\lesssim 1$ for any $W\in G(3,2)$.

For $\xi\neq 0$ let $Span(\xi)=\{k\xi: k\in \mathbb{F}_{p}\}$. Then   
\[
\{L\in G_{S_{1}}: \xi \in Per(L)\}=G_{S_{1}}\cap Per(Span(\xi)).
\]
The rank-nullity theorem implies that $\dim Per(Span(\xi))=2$.  Thus $Per(Span(\xi))\in G(3,2)$, and hence we obtain 
\[
|\{L\in G_{S_{1}}: \xi \in Per(L)\}|\lesssim 1.
\]
Applying estimate \eqref{eq:argument} and Lemma \ref{lem:abstract} (2) with $\beta=1, m =2$, we finish the proof.
\end{proof}  

Note that the above example $S_{1}$  can be considered as a finite fields version of curve 
\[
\Gamma=\{\frac{1}{\sqrt{2}}(\cos t, \sin t, 1): t\in [0, 2\pi])\} \subset \RR^{3}. 
\]
For more details on restricted families of projections with respect to $\Gamma$ we refer to \cite{KOV}, \cite{OV}. In the following, we show  a finite fields version of  curve 
\[
\{(t,t^{2},\cdots, t^{n}): t\in [0,1]\}\subset \RR^{n}. 
\]

\begin{example}\label{ex:ex}
Let $S=\{( a, a^{2} \cdots, a^{n}): a \in \FF_{p}\backslash \{0\}\}$. Then  $|G_{S}|=p-1$ and for any subset  $E\subset \F$,
\[
|\{L\in G_{S}: |\pi^{L}(E)|\leq N\}|\lesssim |G_{S}|N(|E|^{-1}+p^{-(n-1)}).
\]
\end{example}
\begin{proof}
For $n=2$ we have $|G_{S}|\approx |G(2,1)|\approx p$, and the claim  follows by applying Theorem \ref{thm:ChenP}. In the following we fix $n\geq 3$ and let $p$ be a large prime number.

For any $\xi\neq 0$,  
\[
\{L\in G_{S}:\xi \in Per(L) \}=G_{S}\cap Per(Span(\xi)).
\]
The rank-nullity theorem implies that $\dim Per(Span(\xi))=n-1$. 
Observe that any $n$ elements of  $S$ form a nonsingular Vandermonde  matrix, and hence these $n$ vectors are linear independent. It follows that for any hyperplane $W\in G(n,n-1)$,
\[
|W\cap S|\leq n-1\lesssim_{n} 1.
\]
Therefore we obtain
\[
|\{L\in G_{S}: \xi \in Per(L) \}|\lesssim_{n}1.
\] 
Applying estimate \eqref{eq:argument} and  Lemma \ref{lem:abstract} (2) with $\beta=1, m=n-1$, we finish the proof.
\end{proof}

By the special choices of $N$ in the above two examples, we conclude that Marstrand-Mattila type projection theorem hold for the restricted families $(\pi^{L})_{L\in G_{S_{1}}}$ and $(\pi^{L})_{L\in G_{S}}$.


\end{document}